\documentclass[reqno]{amsart}
\usepackage{graphicx,amsaddr,amsmath,amsfonts,amssymb,bm,amsthm,mathtools} 
\newtheorem{definition}{Definition}
\newtheorem{theorem}{Theorem}
\newtheorem{lemma}{Lemma}
\newtheorem{problem}{Problem}
\newtheorem{corollary}{Corollary}
\newtheorem{proposition}{Proposition}
\newtheorem{example}{Example}

\DeclareMathOperator*{\argmin}{arg\,min}

\title{Asymptotics of solutions to the linear search problem}
\author{Robin A.\ Heinonen}
\address{MaLGa \& DICCA, University of Genoa, Genoa, IT}
\date{\today}

\begin{document}

\begin{abstract}
    The exact leading asymptotics of solutions to the symmetric linear search problem are obtained for any positive probability density on the real line with a monotonic, sufficiently regular tail. A similar result holds for densities on a compact interval. 
\end{abstract}

\maketitle

\section{Introduction}
How should one move on a line in order to find, in minimum expected time, an unseen target whose position was drawn from a known probability density? This is the question asked by the classical ``linear search problem'' (LSP) \cite{bellman1963,beck1964,bruss1988,alpern2003}, introduced independently by R.\ Bellman and A.\ Beck in the early 1960s. 

In the literature, a large share of the attention on the LSP has been focused on the case where the probability density is symmetric about the searcher's starting point. In such cases, it is not hard to see that the optimal trajectory zigzags back and forth across the starting point, so that the trajectory may be parametrized by a discrete sequence of \emph{turning points} $\{ x_k\}$ with alternating sign. With this in mind and with no loss of generality, we will, throughout this article, denote the turning points by a sequence of \emph{nonnegative} real numbers $\{x_k\}$ (i.e.\ the moduli of the turning points), with the starting point $x_0$ set to 0.

Searching under uncertainty is a central task in many scientific areas--- biology, robotics, computer science, and game theory, to name a few---and the LSP is one of the simplest, most fundamental search problems that can be devised.\footnote{The LSP is also very closely related to the 2-lane cow-path problem in computer science \cite{kao1996}.} However, despite its simplicity, the LSP is more difficult than it may seem at first glance, and relatively little is known in general about optimal search strategies. For example, no known algorithm can compute the optimal strategy for general probability densities (although dynamic programming can be used to compute an $\epsilon$-optimal solution for any desired accuracy $\epsilon$ \cite{alpern2003}).

Moreover, to our knowledge, there are essentially no quantitative results about the optimal strategies for general probability densities. Most of what \emph{is} known may be summarized by a few special cases, studied by A.\ Beck (1984,1986) \cite{beck1984,beck1986}:
\begin{itemize}
    \item the uniform distribution on a symmetric, compact interval, wherein the optimal strategy is simply to visit the endpoints;
    \item the triangular distribution on the same interval, whose optimal trajectory never terminates (and for which some numerical values of the $x_k$ were computed);
    \item and the normal distribution, for which we have $x_k \sim \sqrt{2 k \log k}.$
\end{itemize}

In the present work, we close this latter gap definitively by rigorously establishing asymptotics for the optimal turning points for \emph{any} (eventually) monotonic, nonzero density on the real line, under reasonable regularity assumptions. The result may be expressed compactly in terms of the hazard function associated with the target probability density. A similar result can be used to evaluate the asymptotics for many densities on compact intervals. In the cases of both bounded and unbounded domains, we also establish rigorously the asymptotics of the turning points in the case where $p$ decays like a power law (up to undetermined constants), which will serve as an important edge case. 

The remainder of the paper is organized as follows. In Sec.~\ref{sec:problem}, we state the LSP precisely along with some useful, well-known results and a few necessary definitions. Then, in Sec.~\ref{sec:results} we state our main results. A few interesting special cases are presented in Sec.~\ref{sec:examples}, including the aforementioned triangular distribution and the normal distribution (as a special case of general stretched/compressed-exponential tails). Finally, the main results are proved in Secs.~\ref{sec:thm1}--\ref{sec:prop5}, and we conclude with a brief discussion in Sec.~\ref{sec:discussion}.

\section{Preliminaries}\label{sec:problem}
\subsection{Problem statement and known results}
Formally, the symmetric LSP may be stated as follows:
\begin{problem}[Symmetric linear search problem] \label{prob:lsp}
Fix a probability density $p(x)$ on $\mathbb{R}$ such that $p$ is even (i.e., $p(x) = p(-x)$ for all $x\in \mathbb{R}$). Select a target $x^* \sim p.$ Let $\Gamma_1$ be the set of continuous, piecewise $\mathcal{C}^1,$ unit-speed curves $\gamma:[0,\infty) \to \mathbb{R}$ with $\gamma(0) = 0.$ Find $\gamma^* := \argmin_{\gamma \in \Gamma_1} \mathbb{E}_{x^* \sim p}[\inf \{ t: \gamma(t) = x^*\}].$
\end{problem}

 Much of what we know about Problem~\ref{prob:lsp} was established by A.\ Beck in a long series of colorfully-named papers \cite{beck1964,beck1965,beck1970,beck1973,beck1984,beck1986,beck1992}.\footnote{``The linear search problem: electric boogaloo'' was a working title for the present manuscript.} For one, as previously stated, the candidate $\gamma$ may be parametrized by a \emph{discrete} sequence of turning points $\{x_k\}_{k\ge 0},$ and WLOG we can take $x_0=0$ and $x_k \ge0$ for all $k\ge1$. We also have the following important, well-known facts, which we state without proof:

\begin{proposition}\label{prop:lspfacts}
Let $\{x_k\}$ be a sequence of turning points parametrizing a minimizing search strategy for Problem~\ref{prob:lsp}. Then the following hold:
\hfill
\begin{enumerate}
        \item Such a minimizing strategy $\{x_k\}$ exists if and only if $\int_0^\infty x p(x) \, dx < \infty.$
    \item Define the survival function $G(x) = \int_x^\infty p(t) \, dt.$ Then $\{x_k\}$ minimizes the objective function 
    \begin{equation}\label{eq:objective}
        J[x] := \sum_{k = 1}^\infty x_k ( G(x_k) + G(x_{k-1})) \tag{Obj}
    \end{equation} and in particular we have (by differentiating) the recurrence
    \begin{equation} \label{eq:recurrence} (x_k+x_{k+1}) p(x_k) = G(x_k) + G(x_{k-1}). \tag{Rec}\end{equation}
    \item We have $x_{k+1} > x_k \; \forall k\ge0$ and, if $p>0$ everywhere, $x_k \to \infty$ as $k \to \infty.$

\end{enumerate}
\end{proposition}
The form of the objective function $J$ follows from a straightforward computation of the expectation of the first passage time (note that a $p$-dependent constant term has been omitted). The recurrence (\ref{eq:recurrence}) essentially reduces the computation of the turning points to finding $x_1,$ but this is known to be extremely difficult. The third fact---that the optimal turning points are strictly increasing in modulus---is nontrivial and very useful. 

It is also obvious that any optimal strategy must visit the entirety of the support of $p$, or else the expected first passage time diverges; for example, for $p>0$ on $\mathbb{R},$ we must have that $x_k \to \infty$ as $k \to \infty.$ This latter case will be the main focus of what follows; due to symmetry, it will generally suffice to think of $p$ as a density on $\mathbb{R}^+.$

\subsection{The hazard function}
We find it useful to perform our analysis in terms of the \emph{hazard function}, a standard object in the theory of probability densities over $\mathbb{R}^+$ and survival analysis \cite{kalbfleisch2002}.

\begin{definition}
Let $p$ be a probability density on $\mathbb{R}^+,$ and let $G(x) := \int_x^\infty p(t) \, dt$ be its survivor function. The function $h(x) := p(x) / G(x)$ is called the \emph{hazard function} of $p.$  
\end{definition}

Equivalently, if $X$ is the random variable specifying the (modulus of the) target position, we have
\[
h(x) = \lim_{\epsilon\to 0} \frac{\mathrm{Pr}(X \in [x,x+\epsilon) \mid X\ge x)}{\epsilon}.
\]
That is, $h(x)$ represents the instantaneous rate per unit distance that the target is located at position $x,$ given that it is not located closer to the origin than $x.$

 Notably, the hazard of a density $p$ uniquely specifies $p$; in particular, since $h = p/G = -G'/G,$ we have $G(x) = \exp(-\int_0^x h(t) \, dt).$ We present the hazards of a few well-known kinds of probability densities below.
\begin{example}[Pareto]
    If $p(x)\asymp x^{-\alpha}$ for some $\alpha>1,$ then the hazard is $h(x) \sim (\alpha-1)/x.$
\end{example}
\begin{example}[Stretched/compressed exponential]
    If $p(x)\asymp \exp(-(x/a)^b)$ for some $a,b>0,$ then the hazard is $h(x) \asymp x^{b-1}.$ In particular, the exponential distribution has constant hazard, and the normal distribution has linear hazard.
\end{example}
\begin{example}[Lognormal]
    If $p(x)\asymp x^{-1} \exp(-\log^2 x/2\sigma^2)$ for some $\sigma>0,$ then the hazard is $h(x) \asymp \log x/x.$
\end{example}
\begin{example}[Gumbel]
If $p(x) \asymp \exp(-\exp(x)),$ then $h(x) \asymp \exp(x).$
\end{example}

It is clear from the definition that the hazard must itself be positive wherever $p>0.$ However, for the kinds of distributions in which we are interested (monotone, sufficiently regular), much more can be said. In particular, we have the following useful Lemma, which establishes a sharp lower bound on $h.$
\begin{lemma}\label{lemma:hazard}
Let $p>0$ be a probability density on $\mathbb{R}^+$ such that $\int_0^\infty x p(x) \, dx < \infty .$ Let $h$ be the hazard function associated with $p.$ If $h$ is eventually monotone and $L := \lim_{x\to \infty} x h(x)$ exists, then $h(x) \ge c/x$ for some $c>1$ and large enough $x.$
\end{lemma}
\begin{proof}
If $h$ is eventually increasing, say for all $x>x_0,$ then $h(x)\ge h(x_0)>0$ and for any $c>0$ we have $x\ge c/h(x_0)$ for large enough $x.$ Therefore $h(x) \ge c/x$ eventually.

Now suppose $h$ is eventually decreasing. If $L=\infty,$ we have immediately that $h(x) = \omega(1/x)$ by definition. Otherwise take $L \in (0,\infty).$ Then $h(x) \sim L /x$ and 
\[ \log G(x) = -\int_{0}^x h(t) \, dt = - L \log x + o(\log x)\]
so $G(x) = x^{-L+o(1)}.$
But the first moment is
\[
\int_0^\infty x p(x) \, dx = \int_0^\infty G(x) dx < \infty
\]
so $L >1.$ Taking any $c\in(1,L)$ then gives the claim for large enough $x.$
\end{proof}

After dividing through by $G(x_k)$ and taking logarithms, the recurrence (\ref{eq:recurrence}) can be re-expressed in terms of the hazard as follows:
\begin{equation}
\label{eq:hazardrecurrence}
    \log\left((x_k+x_{k+1})h(x_k) -1\right) = \int_{x_{k-1}}^{x_k} h(x) \, dx. \tag{Rec'}
\end{equation}
Through this equation, controlling how, and by how much, the hazard can change over the interval $[x_{k-1}, x_k]$ will provide powerful control over the optimal turning points themselves. This motivates the use of monotonicity and regularity hypotheses on the hazard, which are weak assumptions in the sense that they apply to essentially any commonly studied distribution on $\mathbb{R}^+$ with finite expectation.

\subsection{Regular variation and de Haan's class $\Gamma$}
In order to control the integral on the RHS of (\ref{eq:hazardrecurrence}), we need the hazard to be sufficiently regular. The precise necessary notions of regularity are standard \cite{bingham1989} and stated below.\footnote{It is probably possible to generalize and streamline our results by using a unified notion such as ``Beurling regular variation'' \cite{bingham2014}, but instead we choose to work with more well-known notions.}
 \begin{definition}[Slow variation]\label{def:slowvar}
    Let $L: [A,\infty) \to \mathbb{R}^+$ for some $A>0$ be Lebesgue-measurable. We say that $L$ is \emph{slowly varying} if, for all $\lambda>0$
    \[
    L(\lambda x) \sim L(x)
    \]
    as $x\to \infty.$
\end{definition}
Morally speaking, slowly varying functions are slower than polynomial; a prototypical example is any power of a logarithm. As a useful fact, they obey the so-called \emph{Potter bound}, i.e., for any $\eta>0$ and $t_1,t_2$ large enough, there exists $C_\eta$ such that
\[
\frac{L(t_1)}{L(t_2)} \le C_\eta \max \left\{\left( \frac{t_1}{t_2} \right)^\eta,\left( \frac{t_1}{t_2} \right)^{-\eta} \right\}.
\]

Slow variation forms the basis of the notion of ``regular variation.''
 \begin{definition}[Regular variation]\label{def:regvar}
    Let $f: [A,\infty) \to \mathbb{R}^+$ for some $A>0$ be Lebesgue-measurable. We say that $f$ is \emph{regularly varying} of index $\rho$ and $f\in\mathcal{R}(\rho)$ if
    \[
    f(x) = x^\rho L(x)
    \]
    for some slowly varying function $L.$
\end{definition}
In particular, any function with a power-law tail is regularly varying. A standard, useful consequence of regular variation is the \emph{uniform convergence property} which says that, for any $\lambda>0,$
\[
\frac{f(\lambda x)}{f(x)} \to \lambda^\rho
\]
uniformly in $\lambda.$ 

Regular variation will help us characterize the solutions for polynomially growing or slower hazard. In order to extend the results to rapidly varying (say, exponentially growing) functions, we also need an alternate kind of regularity, which is also standard:
\begin{definition}[De Haan's class $\Gamma$]\label{def:gamma}
    Let $f: [A,\infty) \to \mathbb{R}^+$ for some $A>0$ be Lebesgue-measurable. We say that $f$ is in \emph{de Haan's class} $\Gamma$ and $f \in \Gamma$ if there exists an \emph{auxiliary function} $a(x)>0$ such that for all fixed $t\in\mathbb{R},$
    \[
    \frac{f(x + t a(x))}{f(x)} \to e^t
    \]
    as $x \to \infty.$
\end{definition}
As a remark, it is well-known that any such auxiliary function necessarily satisfies $a(x) = o(x)$ (a fact which we will use more than once) and moreover is \emph{self-neglecting}, that is 
\[
\frac{a(x + ta(x))}{a(x)} \to 1,
\]
uniformly in $t.$ 

Note in particular that if $f(x)=\exp(g(x)),$ with $g$ increasing and twice-differentiable and $g''/g'^2  \to 0,$ then $f \in \Gamma$ with auxiliary $1/g'.$ Hence, $\Gamma$ encompasses essentially any ``nice'' function growing faster than any polynomial. Conversely, one can show that any $f\in \Gamma$ grows faster than polynomial (see proof of Theorem~\ref{thm:asymptotics}).

We close this section with a useful fact:  $p$ inherits monotonicity from $h$ if $h$ is either RV or in class $\Gamma.$
\begin{proposition}
Let $p>0$ be a density on $\mathbb{R}^+$ with hazard $h.$ If $h$ is eventually monotone and either regularly varying or in de Haan's class $\Gamma,$ then $p$ is eventually monotone decreasing.
\end{proposition}
\begin{proof}
First consider $h \in \mathcal{R}(\rho).$ First consider $\rho<0,$ so that $h$ is eventually monotone decreasing. Note that $p(x) = h(x) \exp\left(-\int_0^{x} h(t) \, dt \right)$, so for any $\lambda>1$
\[
\frac{p(\lambda x)}{p(x)} = \frac{h(\lambda x)}{h(x)} \exp\left(-\int_x^{\lambda x} h(t) \, dt\right) \le \frac{h(\lambda x)}{h(x)} \to \lambda^{\rho},
\]
with the last step coming from uniform convergence. Hence for any $\epsilon>0,$
\[
\frac{p(\lambda x)}{p(x)} \le \lambda^{\epsilon+\rho},
\]
and taking $\epsilon < |\rho|$ yields $\frac{p(\lambda x)}{p(x)}<1$ for large enough $x.$ 

Now consider $\rho>0$ ($h$ eventually increasing). 
From the Potter bound, we have for any $\lambda>1$ and $\eta>0, $
\[
\frac{h(\lambda x)}{h(x)} \le C_\eta \lambda^{\rho+\eta}
\]
for sufficiently large $x.$ Moreover, $\int_{x}^{\lambda x} h(t) \, dt \ge (\lambda -1 ) x h(x).$ Combining these, we have
\[
\frac{p(\lambda x)}{p(x)} \le C_\eta \lambda^{\rho + \eta} e^{-(\lambda-1) x h(x)} < 1
\]
for large enough $x.$

Finally, we turn our attention to $h\in\Gamma.$ In this case, $h$ is eventually increasing, and we have for any $t>0$
\[
\int_x^{x+ta(x)} h(u) \, du \ge t a(x) h(x),
\]
so
\[
\frac{p(x + t a(x))}{p(x)} \le \frac{h(x+t a(x))}{h(x)} e^{-t a(x) h(x)} \to e^{t(1- a(x) h(x))}.
\]
It remains to show that $a(x)h(x) \to \infty.$ For any $t>0,$ and $\epsilon \in(0,1),$ $h(x+a(x))/h(x) \ge e^{1-\epsilon}$ and (due to the self-neglecting property) $a(x+a(x))/a(x) \ge 1-\epsilon$ for large enough $x.$ Then, defining a sequence $\{x_n\}$ by $x_{n+1} = x_n + a(x_n),$ we have 
\[
\frac{a(x_{n+1})h(x_{n+1})}{a(x_{n})h(x_{n})} \ge (1-\epsilon)e^{1-\epsilon} >1
\]
for small enough $\epsilon.$ Hence $a(x_n) h(x_n) \to \infty.$ We can extend this to any $x$ by choosing $n$ such that $x\in[x_n,x_{n+1}],$ which gives $a(x) \asymp a(x_n)$ (by self-neglecting) and hence $a(x) h(x) \gtrsim a(x_n) h(x_n) \to \infty.$
\end{proof}

\section{Results}\label{sec:results}
We first prove the following lemma, which essentially says that the solutions to the LSP cannot grow faster than exponential under monotone hazards. To establish this control, it is necessary to use the objective function itself (rather than just the recurrence). The proof strategy, which involves introducing a ``competitor'' sequence $y_k,$ previously was used, e.g.\ in \cite{beck1984} when studying the special case of a normally distributed target.
\begin{lemma}\label{lemma:noexplode}
Let $\{x_k\}$ be a minimizing sequence of turning points for Problem~\ref{prob:lsp}. Under the same hypotheses on $h$ and $p$ as in Lemma~\ref{lemma:hazard}, there exists $M>1$ such that $x_{k+1} \le M x_k$ for large enough $k.$
\end{lemma}
\begin{proof}
For any sequence $\{y_k\}$ such that $y_k = x_k \; \forall k\le N,$ we may write
\begin{align*}
J[y] - J[x] &= \sum_{k\ge N} \left( (y_k + y_{k+1}) G(y_k) - (x_k + x_{k+1}) G(x_k) \right) \\ &\le (y_N + y_{N+1}) G(y_N) - (x_N + x_{N+1}) G(x_N) + \sum_{k\ge N+1} (y_k + y_{k+1}) G(y_k) \\ &= (y_{N+1}- x_{N+1}) G(x_N)+ \sum_{k\ge N+1} (y_k + y_{k+1}) G(y_k).
\end{align*}
But by optimality of $x_k,$ $J[y] \ge J[x],$ so
\begin{equation}\label{eq:competition}
x_{N+1} \le y_{N+1} + \frac{1}{G(x_N)} \sum_{k\ge N+1} (y_k + y_{k+1}) G(y_k).
\end{equation}
Now, by Lemma~\ref{lemma:hazard}, there exists $c>1$ and $x_0>0$ such that $h(x) \ge c/x$ for all $x>x_0.$ Then $u\ge v \ge x_0$ implies
\begin{equation} \label{eq:powerbound}
\log \frac{G(u)}{G(v)} = -\int_v^u h(t) \, dt \le -\int_v^u \frac{c}{t} \, dt = -c \log \frac{u}{v} \implies G(u) \le G(v) \left(\frac{u}{v}\right)^{-c}.
\end{equation}
Take $N$ large enough that $x_N \ge x_0,$ and set $r= 2^{1/c} \in (1,2).$ Define $y_{N+j} = x_N r^j$ for $j\ge1.$ Then combining Eq.~\ref{eq:competition} and Eq.~\ref{eq:powerbound} yields
\[
x_{N+1} \le r x_N +  \sum_{j=1}^{\infty} x_N r^j (1 + r)r^{-cj} = r x_N + \frac{x_N(1+r)}{r^{c-1}-1},
\]
so $x_{N+1} \le M x_N$ with $M = (r^c+1)/(r^{c-1}-1) = 3/(2r^{-1}-1) \in(1,\infty).$
\end{proof}
The Lemma can be strengthened if we know that $h(x) = \omega(1/x).$
\begin{corollary}\label{cor:noexplode}
If, in addition to the hypotheses of Lemma~\ref{lemma:noexplode}, we have $h(x) = \omega(1/x),$ then for any $M>1,$ there exists $K$ such that $x_{k+1} \le M x_k$ whenever $k\ge K.$
\end{corollary}
\begin{proof}
    Same as the proof of Lemma~\ref{lemma:noexplode}, \emph{mutatis mutandis}.
\end{proof}

We now turn our attention to the increments $\Delta_k := x_k-x_{k-1}.$ Under an additional weak regularity assumption on $h,$ we can establish a trichotomy which classifies the limiting value of these increments (zero, finite, or infinite). 
\begin{theorem}\label{thm:increments}
Let $p>0$ be a probability density on $\mathbb{R}^+$ such that $\int_0^\infty x p(x) \, dx < \infty .$ Let $h$ be the hazard function associated with $p,$ and let $\{x_k\}$ be a minimizing sequence of turning points for Problem~\ref{prob:lsp} (using the symmetric extension $p(x) = p(-x)$ as the target density). Define $\Delta_k := x_k-x_{k-1}.$ If $h$ is eventually monotone and $L := \lim_{x\to \infty} x h(x)$ exists, then the following hold: 
\hfill
\begin{enumerate}
    \item If $h(x) = o(\log x),$ then $\Delta_k \to \infty.$
    \item If $h(x) = \omega(\log x),$ then $\Delta_k \to 0.$
    \item If $h(x) \sim c \log x$ for some $c\in(0,\infty),$ then $\Delta_k \to 1/c.$
\end{enumerate}
\end{theorem}
In particular, asymptotically logarithmic hazards---equivalent to densities with tail $p\asymp \exp(-c x \log x + O(x))$---act as a boundary case where the optimal turning points grow linearly. This is useful for certain applications which introduce an inequality constraint to the LSP, see for example \cite{heinonen}.

While the increments are interesting objects in their own right, the central result of this article is that, with a bit of extra regularity on $h$ (in the sense defined in the previous section), studying them in fact allows us to characterize the exact leading asymptotics of $x_k.$ We must exclude the case of power law tails ($h(x) \sim 1/x$), wherein LSP behaves qualitatively differently. This is summarized in the following Theorem.
\begin{theorem}\label{thm:asymptotics}
Suppose, in addition to the hypotheses of Theorem~\ref{thm:increments}, that $h(x)= \omega(1/x)$ and either $h\in \mathcal{R}(\rho)$ for some $\rho\in\mathbb{R}$ or $h \in \Gamma$. Then
\begin{equation}\label{eq:delta_asymp}
\Delta_k = \frac{\log(2 x_k h(x_k))}{h(x_k)} + o\left(\frac{1}{h(x_k)}\right)
\end{equation}
and
\begin{equation}\label{eq:int}
k \sim \int^{x_k} \frac{h(x)}{\log (x h(x))}\, dx.\footnote{We are unable to characterize the error term in Eq.~\ref{eq:int} without finer-grain control on $h.$ }
\end{equation}
\end{theorem}

Theorem~\ref{thm:asymptotics} characterizes the asymptotics of any (well-behaved) density which decays faster than a power law; for any explicit expression for $h,$ we can in principle compute the integral (\ref{eq:int}) and invert to find the leading behavior of $x_k.$ For completeness, we also work out the case of power-law (Pareto) tails, in which case the optimal turning points grow exponentially.\footnote{This result is likely known, but it is unclear if it has ever been published.}
\begin{proposition}\label{prop:pareto}
     Suppose $h(x) \sim a/x$ with $a>1,$ with $h$ eventually monotone. Then under optimality,
    \[x_k =r^{k+o(k)},\] where $r$ is the unique solution to 
    \begin{equation}\label{eq:fp}
        r^a=a(r+1)-1
    \end{equation}
    such that $r>1.$
    Moreover, if $h(x) = a/x + O(x^{-(1+\epsilon)})$ for some $\epsilon>0,$ then $x_k \asymp r^k.$
\end{proposition}
The proof is deferred to Sec.~\ref{sec:prop2}.

\subsection{Extension to compact intervals}
If $p$ is instead supported only on a compact interval, say $[-1,1]$ WLOG, the LSP may seem qualitatively different, at least \emph{prima facie}. However, a result closely analogous to Theorem~\ref{thm:asymptotics} holds and characterizes the asymptotics as $x \to 1$ for regular and monotonic hazards. A subtlety is that we must first establish that the sequence of optimal turning points does not terminate; the question of whether or not the sequence terminates was already answered in Ref.~\cite{baston1995} (indeed, in greater generality than shown here), but we present the following simple results for clarity.

\begin{proposition}\label{prop:nontermination}
Let $p>0$ be a symmetric, continuous probability density on $(-1,1)$ such that $\lim_{x \uparrow 1} p(x) =0.$ If $p$ is monotone on some interval $[A,1)$ with $A\in (0,1),$ then the sequence $x_k$ does not terminate (i.e.\, $x_k <1 \, \forall k$).
\end{proposition}
\begin{proof}
    Suppose the contrary; then $x_{N-1} = x_N = 1$ for some $N.$ Consider a second sequence $\{y_k\}$ terminating at $N+1$ with $y_k = x_k$ for $0 \le k \le N-2,$ $y_{N-1} = t > x_{N-2},$ and $y_{N} = y_{N+1} = 1.$ Then
\[
J[y] - J[x] =  (t-1) G(x_{N-2}) + (t + 1) G(t) \ge 0,
\]
so for any $t\in[0,1),$ $G(t)/(1-t) \ge G(x_{N-2})/(1+t).$ 
But for $t$ close enough to 1, $p$ is monotonically decreasing and thus we can take $G(t)/(1-t) \le p(t)$ and hence
\[
p(t) \ge G(x_{N-2})/(1+t) > \frac{1}{2}  G(x_{N-2}).
\]
But we can always take $t$ close enough to 1 that this is false, contradiction.
\end{proof}

On the other hand, if $p$ does not go to 0 at the boundary of the interval, it is easy to show that the optimal search strategy reaches the boundary in \emph{finite} time, and it does not make sense to talk about large-$k$ asymptotics. 

\begin{proposition}\label{prop:termination}
Let $p>0$ be a symmetric, continuous probability density on $(-1,1)$ such that $L:=\lim_{x \uparrow 1} p(x) >0.$ Then there exists $N \in \mathbb{N}$ such that $x_N = 1.$
\end{proposition}
\begin{proof}
Assume the contrary, so $x_k <1$ for all $k.$ Then $x_k \uparrow 1$ as $k \to \infty.$ Since $p$ is continuous, $G(x) = \int_x^1 p(t) \, dt \to 0.$ But
\[
(x_k + x_{k+1}) p(x_k) = G(x_k) + G(x_{k-1}),
\]
and the LHS$\to 2L>0$ while the RHS$\to 0,$ contradiction.
\end{proof}

If the hazard is sufficiently regular and goes to infinity at the boundary sufficiently quickly, we have the following theorem, which we prove in Sec.~\ref{sec:thm3}.
\begin{theorem}\label{thm:bounded}
Let $p>0$ be a symmetric, continuous probability density on $(-1,1)$ such that $\lim_{x \uparrow 1} p(x) =0.$ Let $G(x) = \int_{x}^1 p(t) \, dt$ and $h(x) = p(x) /G(x).$ Let $\{x_k\}$ be an optimal strategy for Problem~\ref{prob:lsp}, and let $\Delta_k := x_k - x_{k-1}.$ If, on some interval $[A,\infty)$ with $A>0,$ $\tilde h(x) := h(1-1/x)$ is monotone and either regularly varying of index $\rho > 1$ or in de Haan's class $\Gamma,$ then the asymptotics Eq.~\ref{eq:delta_asymp}-\ref{eq:int} hold.
\end{theorem}

Once again, we find that $h(\epsilon) \asymp 1/\epsilon$ (with $\epsilon =1-x$) represents a boundary case which must be treated differently than the others. To study this case, it is illuminating to first establish the following simple Lemma, the analog to Lemma~\ref{lemma:hazard}:
\begin{lemma}
Let $p$ be a continuous and eventually monotone density on $(0,1)$ with $p\to 0$ as $x \uparrow 1.$ Then the hazard obeys $h(x) \ge 1/(1-x)$ for $x\in(0,1)$ close enough to 1.
\end{lemma}
\begin{proof}
    Since $p\to 0$ and is positive, $p$ is monotone decreasing on $[A,1)$ for some $A>0.$ Then for all $x\in[A,1),$
    \[
    G(x) = \int_x^1 p(t) \, dt \le (1-x) p(x)
    \]
and so $h(x) = p(x)/G(x) \ge 1/(1-x).$
\end{proof}

Note the important distinction: due to the fact that finite expectation is not a useful control for densities on compact intervals, we are unable to guarantee that $h(x) \ge c/(1-x)$ with $c>1$ \emph{strictly}. In the event that $c>1,$ which corresponds to densities going to 0 like a power law at the boundary, we can prove the following proposition, showing that the residual $1-x_k$ decays doubly exponentially.
\begin{proposition}\label{prop:bounded_powerlaw}
Let $p>0$ be a symmetric, continuous probability density on $(-1,1).$ If the hazard $h$ satisfies $h \sim c/(1-x)$ as $x \to 1$ for some $c>1,$ then 
\[\log(1- x_k) \asymp -\left(\frac{c}{c-1}\right)^k.
\]
If, furthermore, $h(x)= c/(1-x) + O((1-x)^{\delta -1})$ for some $\delta>0,$ then 
\[
1-x_k \sim 2c \exp\left(-A \left(\frac{c}{c-1}\right)^k\right)
\]
for some $A>0.$
\end{proposition}

The proof of Proposition~\ref{prop:bounded_powerlaw} is given in Sec.~\ref{sec:prop5}. However, the boundary case where $h(x) \sim 1/(1-x)$ remains difficult to describe in generality; this corresponds to $p(1-\epsilon) \sim \ell(1/\epsilon)$ for some $\ell$ slowly varying. In this case, the asymptotics of the optimal turning points appear to depend sensitively on the next-to-leading behavior of the hazard. We defer the classification of such cases to future study.

\section{Examples}\label{sec:examples}
Eq.~\ref{eq:int} allows us find the asymptotics of $\{x_k\}$ for essentially \emph{any} well-behaved, eventually monotonic hazard on the real line, except for the degenerate case $h(x) \asymp 1/x$ (power-law tails), which was handled in Proposition~\ref{prop:pareto}. Here are a few interesting examples which do not appear to be known in the literature; they follow from straightforward calculations, whose details we have suppressed for the sake of brevity.

\begin{example}[Compressed/stretched exponential tails]
Suppose, in addition to the hypotheses of Theorem~\ref{thm:asymptotics}, we have $h(x) \sim ax^b$ for any $a>0,b>-1.$ Then under optimality, \[x_k \sim \left(\frac{1+b}{a} k \log k\right)^{1/(1+b)}.\]
\end{example}
Note that the above subsumes as special cases both the exponential ($b=0$) and normal ($b=1$) distributions; the asymptotics of $\{x_k\}$ in the latter case were one of the the main results of Ref.~\cite{beck1984}.

\begin{example}[Lognormal tails]
Suppose, in addition to the hypotheses of Theorem~\ref{thm:asymptotics}, we have $h(x) \sim a \log x/x$ for any $a>0.$ Then under optimality, \[x_k \sim e^{\sqrt{k \log k/a}}.\]
\end{example}
\begin{example}[Gumbel tails] 
Suppose, in addition to the hypotheses of Theorem~\ref{thm:asymptotics}, we have $h(x)\sim e^{ax}$ for $a>0$ (Unlike the previous examples, this hazard is de Haan $\Gamma$ and not RV.) Then under optimality,
\[ x_k \sim \frac{1}{a} \log k.\]
\end{example}

Let us also present two examples of a distribution on a compact interval. First, consider the triangular distribution. It was previously shown in Ref.~\cite{beck1984} that the optimal search strategy in this case has an infinite number of turning points, but no quantitative asymptotic was given. A quick calculation gives $h(x) = 2/(1-x),$ and we can apply Proposition~\ref{prop:bounded_powerlaw} to find that the distance to the interval boundary decays like a double exponential.
\begin{example}[Triangular distribution]
    Let a symmetric probability density $p>0$ be supported on $(-1,1)$ and suppose the hazard obeys $h(x) \sim 2/(1-x) + O((1-x)^{\delta-1})$ for some $\delta>0.$ Then under optimality,
    \[
    1- x_k \sim  4 e^{-A2^k}.
    \]
    for some $A>0.$
\end{example}
Finally, we consider tails decaying like $p(1-\epsilon) \asymp \exp(- a/\epsilon^c + o(1/\epsilon)),$ $c>0.$ Due to the rapid decay of $p$ at the boundary, $1-x_k$ decays much more slowly in this case.
\begin{example}[A fast-decaying family of distributions on $(-1,1)$]
Let a symmetric probability density $p>0$ be supported on $(-1,1)$ and suppose the hazard obeys $h(x) \sim \frac{a}{(1-x)^{1+b}}$ for some $a,b>0.$ Then under optimality,
\[
1-x_k \sim \left(\frac{1+b}{a} k \right)^{-1/b}.
\]
\end{example}

\section{Proof of Theorem \ref{thm:increments}}\label{sec:thm1}
Our proof strategy is similar to that of Lemma~\ref{lemma:noexplode}, in that we construct a competitor sequence which upper bounds the growth of the increments, this time in an $h$-dependent manner. 
\begin{proof}
From (\ref{eq:recurrence}), we have by monotonicity that
\begin{equation} \label{eq:monotone}
    \underline{h}_k := \min \{ h(x_{k-1}),h(x_k) \}  \le \frac{ \int_{x_{k-1}}^{x_k} h(x) \, dx }{\Delta_k}\le  \max \{ h(x_{k-1}),h(x_k) \} := \bar{h}_k.
\end{equation}
Thus we have
\begin{equation}\label{eq:delta_ineq}
\Delta_k \ge \frac{\log\left((x_k+x_{k+1})h(x_k) -1\right)}{\bar{h}_k} \ge \frac{\log\left(2 x_k h(x_k) -1\right)}{\bar{h}_k} \ge \frac{\log\left(x_k h(x_k)\right)}{\bar{h}_k},
\end{equation}
where the last inequality follows from Lemma~\ref{lemma:hazard}. 

We first focus on the case $h(x) = o(\log x).$ If $x_k \ge 2 x_{k-1},$ then $\Delta_k = x_k-x_{k-1} \ge x_k \to \infty.$ On the other hand, if $x_k < 2 x_{k-1},$ then $\log x_k = \log {x_{k-1}} + O(1),$ so $\bar h_k = o(\log x_k)$ regardless of whether $h$ is increasing or decreasing. Hence
\begin{equation} \Delta_k \ge \frac{\log x_k +O(1)}{o(\log x_k)} \to \infty. \end{equation}
This proves the first case.

Now suppose that $h$ is increasing, consistent with $h=\Omega(\log x)$. Then using Eq.~\ref{eq:competition}, $u\ge v$ implies 
\begin{equation}
G(u) = G(v) \exp\left(- \int_{v}^u h(x) \, dx \right) \le G(v) e^{-h(u)(u-v)}.
\end{equation}
Set $y_{N+j} = x_N + r j$ for each $j\ge1,$ with $r = \log(x_N h(x_N))/h(x_N).$ We then have
\[
\frac{G(y_{N+j})}{G(x_N)} \le  \exp(-h(x_N)rj) = (x_N h(x_N))^{-j}.
\]
Inserting into Eq.~\ref{eq:competition}, we find 
\begin{align*}
    x_{N+1} - x_N &\le r + \sum_{j=1}^\infty (2x_N + 2rj + r)(x_N h(x_N))^{-j} \\ &= r + \frac{2x_N + r}{x_N h(x_N)-1}+\frac{2x_N h(x_n) r}{(x_N h(x_N) -1 )^2} \\ &=  r + O\left(\frac{1}{h(x_N)}\right).
\end{align*}
But the choice of $N$ was arbitrary, so we have (in combination with Eq.~\ref{eq:delta_ineq} and using $h$ monotonically increasing)
\begin{equation}\label{eq:delta_bounds}
\frac{\log (x_k h(x_k)) }{h(x_k)} \le \Delta_k \le \frac{\log (x_{k-1} h(x_{k-1})) + O(1) }{h(x_{k-1})}.
\end{equation}
The upper and lower bounds both tend to zero if $h(x) = \omega(\log x).$ On the other hand, by Lemma~\ref{lemma:noexplode}, $\log(x_k) = \log {x_{k-1}} + O(1)$, so if $h(x) \sim c \log x,$ the upper and lower bounds are both equal to $1/c + O(\log \log x_k/\log x_k),$ so it follows that $\Delta_k \to 1/c.$
\end{proof}

\section{Proof of Theorem \ref{thm:asymptotics}}\label{sec:thm2}
Equations~\ref{eq:monotone} and \ref{eq:delta_bounds} are already strongly suggestive of the final result. The main ingredient that extra regularity provides is a guarantee that $h(x_{k-1}) \sim h(x_{k})$ through either the uniform convergence property of regular variation or the $o(x)$ property of the auxiliary functions for de Haan's class $\Gamma.$ The validity of integrating the increments (Eq.~\ref{eq:int}) follows by establishing that the increments do not grow too quickly.
\begin{proof}
Using Corollary~\ref{cor:noexplode} and $h(x) = \omega(1/x)$, we have from (\ref{eq:hazardrecurrence}) that
\[
\int_{x_{k-1}}^{x_k} h(x) \, dx = \log(2x_k h(x_k)) + o(1).
\]
Since $h$ is eventually monotone, we can use Eq.~\ref{eq:monotone} to write
\begin{equation} \label{eq:monotone2}
\frac{\log (2x_k h(x_k)) + o(1)}{\bar{h}_k} \le \Delta_k \le \frac{\log (2x_k h(x_k)) + o(1)}{\underline{h}_k}.
\end{equation}

We first show that $\Delta_k = o(x_k).$ If $h$ is decreasing, then $\underline{h}_k = h(x_{k})$ and Eq.~\ref{eq:monotone2} gives
\[
\frac{\Delta_k}{x_k} \le \frac{\log (x_k h(x_k)) + O(1)}{x_k h(x_k)} \to 0
\]
since $x_k h(x_k) \to \infty.$ On the other hand, with $h$ increasing, Eq.~\ref{eq:delta_bounds} still holds and in particular
\[
\frac{\Delta_k}{x_{k}} \le \frac{\Delta_k}{x_{k-1}} \le \frac{\log (x_{k-1} h(x_{k-1})) + O(1) }{x_{k-1} h(x_{k-1})} \to 0. 
\]

Next, we show that if $h\in \mathcal{R}(\rho)$ is regularly varying, then $h(x_k) \sim h(x_{k-1})$ so that $\bar{h}_k \sim \underline{h}_k,$ yielding Eq.~\ref{eq:delta_asymp} from Eq.~\ref{eq:monotone2}. By the uniform convergence property of regularly varying functions, we have for any $u(x) = o(x)$ that
\[
\frac{h(x+u(x))}{h(x)} = \frac{h(x(1+u(x)/x))}{h(x)} \to \left(1+\frac{u(x)}{x}\right)^\rho \to 1.
\]
Using $u(x_k) = -\Delta_k$ gives $h(x_k-\Delta_k)/h(x_k) = h(x_{k-1})/h(x_k) \to 1.$

Now we prove Eq.~\ref{eq:int} for regularly varying hazard. Let $f(x) = \log (x h(x))/h(x).$ it is easy to see that $f$ is regularly varying of index $-\rho.$ The uniform convergence property then gives 
\[
\left|\frac{f(\lambda x)}{f(x)} - \lambda^{-\rho}\right| \le \eta
\]
uniformly for any $\eta >0,$ for large enough $x.$ It follows that there exists $\delta >0 $ such that $|1-\lambda| \le \delta$ implies
\[
1-\eta \le \frac{f(\lambda x)}{f(x)} \le 1+\eta.
\]
Putting $x= x_k$ and $\lambda = t/x_k$ for $t\in[x_{k-1},x_k],$ we have $\lambda \in [1-\Delta_k/x_k,1],$ so since $\Delta_k /x_k \to 0$ we eventually have 
\[
(1-\eta) f(x_k) \le f(t) \le (1+\eta) f(x_k).
\]
It follows that
\[
\frac{\Delta_k}{(1+\eta) f(x_k)} \le \int_{x_{k-1}}^{x_k} \frac{dt}{f(t)} \le \frac{\Delta_k}{(1-\eta) f(x_k)}.
\]
But $\Delta_k/f(x_k) \to 1,$ so taking $\eta \to 0,$ we find $\int_{x_{k-1}}^{x_k} \frac{dt}{f(t)} \to 1.$ Let $F(x) = \int^x \frac{dt}{f(t)}.$ Then for any $k_0<k,$
\[
F(x_k) - F(x_{k_0}) = \sum_{j=k_0}^{k-1} (F(x_{j+1}) - F(x_{j})) \to k-k_0,
\]
so $F(x_k) \sim k.$ This is the desired result. 

Finally, we move to the case where $h \in \Gamma.$ Let $a$ be the auxiliary function. First, we claim that
\begin{equation} \label{eq:claim}
a(x) = \omega\left( \frac{\log h(x)}{h(x)}\right).
\end{equation}
Suppose the contrary, and there is some $C>0$ such that $a(x) \le C \log h(x)/h(x)$ for large enough $x.$ By Def.~\ref{def:gamma} with $t=1,$ we have
$h(x+a(x))/h(x) \to e.$ Define a sequence $y_{n+1} =y_n + a(y_n) $ for $n \ge 0.$ Iterating then gives 
\begin{equation}\label{eq:wrong}
h(y_n) \sim e^n h(y_0).
\end{equation}
Hence,
\[
a(y_n) \le C \frac{\log h(y_n)}{h(y_n)} \sim C \frac{ n}{e^n h(y_0)},
\]
and it follows that $\sum_{n\ge 0} a(y_n) < \infty$ and therefore $y_n$ converges to a finite limit, contradicting Eq.~\ref{eq:wrong}. 

On the other hand, we claim that $h(x)/x^m \to \infty$ for any $m>0.$ Indeed, if $h$ has auxiliary function $a,$ for any $t>1$ we have $ h(x+ta(x)) \ge e^{t} h(x)$ for large enough $x.$ Define $y_{n+1} := y_n + ta(y_n).$ Then
\[
h(y_n) \ge e^{t} h(y_{n-1}) \ge \dots \ge e^{nt} h(y_0).
\]
But $ta(y_n) \le y_n/m$ since $a(x) = o(x),$ so $y_{n+1} \le (1+1/m) y_n$ and hence $y_n \le x_0 (1+1/m)^n.$
It follows that
\[
\frac{h(y_n)}{y_n^m} \ge \frac{h(y_0)}{y_0^m} \left(\frac{e^{t}}{(1+1/m)^m} \right)^n \to \infty
\]
since $e^t > e > (1+1/m)^m$ for any $m.$ The claim follows from $h$ being monotonic.

Hence, $\log x = o(\log h)$ and we can refine Eq.~\ref{eq:claim} further to 
\begin{equation}\label{eq:claim2}
    a(x) = \omega\left( \frac{ \log (x h(x))}{h(x)}\right).
\end{equation}

Next, note that for $r>0$
\begin{align*}
\int_{x - r a(x)}^x h(t) \, dt &= a(x) h(x) \int_0^r \frac{h(x-sa(x))}h(x) \, ds \\ &\to a(x) h(x) \int_0 ^r e^{-s} \, ds  \\ &= (1- e^{-r})a(x) h(x).
\end{align*}

We now claim that $\Delta_k =o(a(x_k)).$ Otherwise, there is a subsequence $x_{k_j}$ and $r>0$ such that $ \Delta_{k_j} \ge r a(x_{k_j}).$ Then 
\[
\frac{1}{a(x_{k_j}) h(x_{k_j})}\int_{x_{k_j-1}}^{x_{k_j}} h(t) \, dt \ge \int_{x_{k_j}-ra(x_{k_j})}^{x_{k_j}} h(t) \, dt \sim 1- e^{-r}
\]
But combining
\[
\int_{x_{k-1}}^{x_k} h(t) \, dt \le \log (x_k h(x_k)) + O(1)
\]
with Eq.~\ref{eq:claim2} yields
\[
\frac{1}{a(x_k) h(x_k)} \int_{x_{k-1}}^{x_k} h(t) \, dt \to 0,
\]
a contradiction.

Thus, setting $t_k = \Delta_k/a(x_k) = o(1),$ we have
\[
\frac{h(x_{k-1})}{h(x_{k})} =  \frac{h(x_{k-1} + t_k a(x_k))}{h(x_{k})} \to 1
\]
by Def.~\ref{def:gamma}. This means we once again have Eq.~\ref{eq:delta_asymp}.

Finally, we have
\[
\frac{h(x-u(x))}{h(x)} \to 1
\]
uniformly whenever $u(x) = o(a(x))$. But $a(x) = o(x),$ so we further have
\[
\frac{\log ((x-u(x))h(x-u(x)))}{\log (x h(x))} \to 1
\]
and so
\[
\frac{f(x-u(x))}{f(x)} \to 1,
\]
with $f(x) = \log (x h(x))/h(x)$ as before. This allows us to reason similarly as before, so that Eq.~\ref{eq:int} holds.
\end{proof}

\section{Proof of Proposition \ref{prop:pareto}}\label{sec:prop2}
We first need to establish existence/uniqueness of the root of Eq.~\ref{eq:fp}. Then, regular variation is enough to establish exponential growth up to a subleading error term; more precise control on the asymptotics of $h$ in turn is enough to establish that this error term is in fact zero.
\begin{proof}
First, we claim that Eq.~\ref{eq:fp} has a unique solution on $(1,\infty).$ To see this, write $F(x) := (x^a+1)/a-1-x,$ so that solutions satisfy $F(x)=0.$ Then $x>1$ implies 
\[
F'(x) = x^{a-1} - 1 >0,
\]
so $F$ is strictly increasing on $(1,\infty).$ Moreover, $F(1) = 2(1-a)/a <0,$ while $F(x) \to \infty$ for $x\to \infty,$ so the claim follows.

Now let $r_k = x_k/x_{k-1}.$ By Lemma~\ref{lemma:noexplode}, $r_k \in (1,C]$ for some $C>1$ eventually. Because $p$ is measurable and $p>0,$ $h$ is also measurable and $h\in \mathcal{R}(-1)$, so the uniform convergence property of regularly varying functions gives 
\[
\sup_{t\in [1,C]} t\left| \frac{h(tx)}{h(x)} - 1/t \right | = \sup_{t\in [1,C]} \left| \frac{h(tx)}{a/(tx)} - 1 \right |\to 0
\]
and thus
\[
\sup_{x\in[x_{k-1},x_k]} |h(x) - a/x| \le \sup_{x\in[x_k/C,x_k]} |h(x) - a/x| = o(1/x_k).
\]
Hence,
\begin{equation}\label{eq:int_asymp}
    \int_{x_{k-1}}^{x_k} h(x) \, dx = a \int_{x_{k-1}}^{x_k} \frac{dx}{x} + o\left(\int_{x_{k-1}}^{x_k} \frac{dx}{x_k}\right) = a\log r_k + o(1).
\end{equation}

On the other hand,
\begin{align*}
(x_k+x_{k+1})h(x_k) -1 &=x_k h(x_k) (1+r_{k+1}) -1 \\ &= (a+o(1))(1+r_{k+1}) - 1 \\ &= a(1+r_{k+1}) - 1 +o(1),
\end{align*}
where the last step used the boundedness of $r_{k+1}.$ Thus, from (\ref{eq:hazardrecurrence}) and Eq.~\ref{eq:int_asymp}, we have
\begin{equation}\label{eq:asymp_recur}
r_{k+1} = \frac{r_k^a+1}{a} -1 + o(1) = g(r_k) +o(1),
\end{equation}
where we have defined $g(x) = (x^a+1)/a-1.$

Next, let $m:= \liminf_{k\to \infty} r_k$ and $M:= \limsup_{k\to\infty}.$ By Lemma~\ref{lemma:noexplode}, $M< \infty.$ We also have $m>1$; if not, then there is a subsequence $k_j$ such that $r_{k_j} \to 1,$ and \begin{equation}\label{eq:subseq}
\int_{x_{k_j-1}}^{x_{k_j}} h(x) \, dx \sim a \log r_{k_j} \to 0.
\end{equation}
But then, using (\ref{eq:hazardrecurrence}) and $xh(x) \to a$,
\begin{align*}
\exp\left(\int_{x_{k-1}}^{x_k} h(x)\, dx\right) &=  x_k h(x_k) (1+r_{k+1})-1 \\ & \to a (1+r_{k+1}) -1 \\ &>\log (2a-1) >1,
\end{align*}
contradicting Eq.~\ref{eq:subseq}.

Now, choose a subsequence $k_j$ with $r_{k_j} \to M.$ From Eq.~\ref{eq:asymp_recur}, $r_{k_j+1} \to g(M).$ But $\limsup_{j\to\infty} r_j \le M,$ so $g(M) \le M,$ and $F(M) \le 0.$ Since $F$ is increasing and $F(r) = 0,$ we conclude $M\le r.$ 

Similarly, along a subsequence with $r_{k_j} \to m,$ we get $r_{k_j+1} \to g(m)$ and \emph{mutatis mutandis} we conclude $m \ge r.$ Therefore, $r_k \to r.$

Finally, since $x_k = x_0 \prod_{j=1}^k r_j,$
\[
\frac1k \log x_k = 1/k \log x_0 + \sum_{j=1}^{k} \log r_j \to \log r
\]
so
\[
x_k = \exp(k(\log r + o(1))) = r^{k + o(k)}.
\]
This proves the first part.

If $h(x) - a/x = O(x^{-(1+\epsilon)}),$
then
\[
\int_{x_{k-1}}^{x_k} h(x) \, dx = a\log r_k + O\left(x_{k-1}^{-\epsilon}\right)
\]
and
\[
(x_k+x_{k+1}) h(x_k) -1 = a(1+r_{k+1}) -1 + O(x_k^{-\epsilon}).
\]
Since $r_{k+1}$ is eventually bounded, we then have
\[
\log ( (x_k+x_{k+1}) h(x_k) -1) = \log (a(1+r_{k+1}) -1 ) + O(x_k^{-\epsilon})
\]
and so
\[
a \log r_k = \log (a(1+r_{k+1}) -1 ) + O(x_{k-1}^{-\epsilon}),
\]
or
\[
r_k^a = a(1+r_{k+1}) -1 + O(x_{k-1}^{-\epsilon}).
\]
Put $G(x) : = (a(1+x)-1)^{1/a}$ and call the error term $\xi_k.$ We have 
\[
G'(x) = (a(1+x)-1)^{1/a-1} = H(x)^{a-1}
\]
so $G'(r) = r^{1-a} \in (0,1),$ so $G$ is a contraction mapping.

Since $G(r) = r$ we have
\[
r_k - r = G(r_{k+1}) - G(r) + \xi_k
\]
and so
\[
|r_k -r| \le \lambda |r_{k+1} -r| + |\xi_k|
\]
with $\lambda \in (0,1).$ It follows that 
\[
|r_k -r| \le \lambda^n |r_{k+n} -r| + \sum_{j=0}^{n-1} \lambda^j|\xi_{k+j}|
\]
for each $n\ge1.$ Taking $n \to \infty,$ we conclude
\[
|r_k -r| \le \sum_{j=0}^{\infty}\lambda^j|\xi_{k+j}|.
\]
We already showed that $\liminf r_k >1,$ so put $r_k \ge 1+\delta$ for some $\delta >0.$ Then $x_{k+j} \ge x_k (1+\delta)^j$ so $|\xi_{k+j}| \le C x_k^{-\epsilon}(1+\delta)^{-j\epsilon}$ for some finite $C>0.$ It follows that $\sum_{j=0}^{\infty}\lambda^j|\xi_{k+j}| \le C' x_k^{-\epsilon}$ for some finite $C',$ and therefore $|r_k - r| = O(x_k^{-\epsilon}). $

Finally, we have
\[
\frac{x_k}{r^k} = x_0 \prod_{j=1}^k \frac{r_j}{r}
\]
which converges iff $\sum_{j=1}^k \log \frac{r_j}{r}$ converges. But $\log (r_j/r) = (r_j-r)/r + O((r_j-r)^2) = O(r_j -r ),$ and since $|r_j-r| = O(x_k^{-\epsilon})$ with $x_{k+j} \ge (1+ \delta)^j x_k$ for large enough $k,$ the series is summable. Hence $x_k \sim C r^k$ for some $C.$
\end{proof}

\section{Proof of Theorem~\ref{thm:bounded}}\label{sec:thm3}
The proof is essentially the same structurally as that of Theorem~\ref{thm:asymptotics}; once again, the main point is to establish that $h(x_{k}) \sim h(x_{k-1})$ using the uniform convergence property, although the fine details differ.
\begin{proof}
First, Proposition~\ref{prop:nontermination} shows that the $x_k$ are non-terminating. Now let $\epsilon = 1-x,$ $z=1/\epsilon.$ 

We first prove that $\Delta_k/(1-x_k) \to 0.$ Suppose, on the contrary, that $\Delta_k \ge \theta \epsilon_k$ for some $\theta \in (0,1)$ along some subsequence. Then on that subsequence, 
\[ x_{k-1} = x_k - \Delta_k \le 1 - (1+ \theta) \epsilon_k,\] so
\begin{equation*} 
\int_{x_{k-1}}^{x_k} h(t) \, dt \ge \int_{1-\epsilon_k }^{1-(1+\theta) \epsilon_k} h(t) \, dt = \int_{\epsilon_k}^{(1+\theta)\epsilon_k}  \tilde h(1/s) \, ds.
\end{equation*}
Since $\tilde h$ is monotone, we have for all $s\in [\epsilon_k,(1+\theta)\epsilon_k]$ that
\[
\tilde h(1/s) \ge \tilde h\left(\frac{z_k}{1+\theta} \right),
\]
so
\[
\int_{x_{k-1}}^{x_k} h(t) \, dt \ge \theta \epsilon_k \tilde h\left(\frac{z_k}{1+\theta} \right),
\]
On the other hand, $x_k + x_{k+1} \le 2,$ so (\ref{eq:hazardrecurrence}) gives
\[
\int_{x_{k-1}}^{x_k} h(t) \, dt \le \log (2 h(x_k)) = \log (2 \tilde h(z_k)),
\]
so 
\begin{equation} \label{eq:theta}
    \log (2 \tilde h(z_k)) \ge \frac{ \theta \tilde h(z_k/(1+\theta))}{z_k}.
\end{equation}

Now suppose that $\tilde h \in \mathcal{R}(\rho).$ Choosing $\lambda = 1/(1+\theta)>0,$ we have 
$\tilde h(\lambda z)/\tilde h(z) \to \lambda^\rho $ as $z\to \infty,$ and hence $\tilde h(z_k/(1+\theta)) \sim (1+\theta)^{-\rho} \tilde h(z_k).$ Using $\tilde h(z) \sim z^\rho L(z)$ (with $L$ slowly varying), we then have
\[
\frac{z_k \log (2\tilde h (z_k))}{\theta \tilde h(z_k/(1+\theta))} \sim \frac{(1+\theta)^\rho z_k \log (2\tilde h (z_k))}{\theta \tilde h(z_k)} = O\left( \frac{\log z_k}{z_k^{\rho-1} L(z_k)} \right) \to 0
\]
since $\rho >1,$ contradicting Eq.~\ref{eq:theta}.

If instead $\tilde h\in \Gamma,$ we again have $\tilde h(z)/z^m \to \infty$ for any $m>0.$ This immediately implies that
\[
\frac{\tilde h(z_k/(1+\theta))}{z_k \log \tilde h(z_k)} \to \infty,
\]
again contradicting Eq.~\ref{eq:theta}.

Now, by monotonicity, Eq.~\ref{eq:monotone} again holds, and since $\log((x_k + x_{k+1}) h(x_k) - 1) = \log(2x_k h(x_k)) + o(1),$ combining with (\ref{eq:hazardrecurrence}) again shows that
\[
\Delta_k = \frac{\log(2x_k h(x_k))+o(1)}{h(x_k)}
\]
as desired (in fact, the factor of $x_k$ in the logarithm can be neglected in this case since $x_k \to 1$). The proof of Eq.~\ref{eq:int} is essentially unchanged from the corresponding part of the proof of Theorem~\ref{thm:asymptotics}.
\end{proof}
\section{Proof of Proposition~\ref{prop:bounded_powerlaw}}\label{sec:prop5}
The proof is structurally similar to that of Proposition~\ref{prop:pareto}: we establish that $G$ is regularly varying, and then use this to derive a recurrence for the logarithm of $1-x_k$, up to $O(1)$ errors. More precise control on the higher-order asymptotics of $h$ suffices to shrink this error to $o(1)$ and refine our estimate of the $x_k.$
\begin{proof}
Let $\epsilon = 1-x$ and $\epsilon_k = 1-x_k.$ We have $d \log G(1-\epsilon)/d\epsilon = h(1-\epsilon).$ Write $h(1-\epsilon) = c/\epsilon + r(\epsilon)$ with $r(\epsilon) = o(1/\epsilon)$ as $\epsilon \downarrow 0.$ Then, integrating, we have
\[
\log G(1-\epsilon) = \log K + c\log \epsilon + \int^\epsilon r(t) \, dt.
\]
for some $K>0$
Hence
\[
G(1-\epsilon) = \epsilon^c \ell(1/\epsilon)
\]
with 
$\ell (t) := K \exp(\int^{1/t} r(u)\, du ).$ We claim that $\ell$ is slowly varying. To see this, note that for any $\lambda >0,$
\[
\log \frac{\ell(\lambda t)}{\ell(t)} = \int_{1/t}^{1/\lambda t} r(u) \, du.
\]
But $r(u) = o(1/u),$ so for any $\eta>0,$ $|r(u)| \le \eta/u$ for small enough $u,$ so
\[
\left|\int_{1/t}^{1/\lambda t} r(u) \, du \right| \le \eta |\log \lambda|.
\]
Taking $t \to \infty$ and $\eta \to 0$ shows that
\[
\log \frac{\ell(\lambda t)}{\ell(t)} \to 0.
\]

We then have 
\[
\frac{G(x_{k-1})}{G(x_k)} = \left(\frac{\epsilon_{k-1} }{\epsilon_k}\right)^c \frac{\ell(1/\epsilon_{k-1})}{\ell(1/\epsilon_k)}.
\]
and so, using $x_k,x_{k+1} \to 1$ and $h(x)\sim c/(1-x),$ (\ref{eq:hazardrecurrence}) gives
\[
\left(\frac{\epsilon_{k-1} }{\epsilon_k}\right)^c \frac{\ell(1/\epsilon_{k-1})}{\ell(1/\epsilon_k)} = \frac{2c}{\epsilon_k}(1+o(1)).
\]
The Potter bound then gives, for any $\eta>0,$
\[
(c-\eta) \log \frac{\epsilon_{k-1} }{\epsilon_k} + O(1) \le \log \frac{2c}{\epsilon_k} \le (c+\eta) \log \frac{\epsilon_{k-1} }{\epsilon_k} + O(1),
\]
so $\eta \to 0$ shows that
\[
\log \frac{2c}{\epsilon_k} = c\log \frac{\epsilon_{k-1} }{\epsilon_k} + O(1),
\]
or, writing $L_k = -\log \epsilon_k,$
\[
L_k = \frac{c}{c-1} L_{k-1} + O(1).
\]
It follows that 
\[ L_k \asymp \left(\frac{c}{c-1}\right)^k  \]
for some $A>0,$ proving the first claim.

On the other hand, if $r = O(\epsilon^{\delta-1}),$ then we can refine our estimate of $G$ further to
\[
G(1-\epsilon) = K \epsilon^c (1+O(\epsilon^\delta)).
\]
Now (\ref{eq:hazardrecurrence}) gives
\[
\left(\frac{\epsilon_{k-1}}{\epsilon_k}\right)^c = \frac{2c}{\epsilon_k}(1+o(1)),
\]
whence
\[
L_k = \frac{c}{c-1} L_{k-1} + \frac{\log 2c}{1-c} + o(1).
\]
It follows that there exists $A>0$ such that
\[
L_k = A\left(\frac{c}{c-1} \right)^k -\log 2c +o(1),
\]
proving the second claim.
\end{proof}

\section{Discussion}\label{sec:discussion}
Taken together, Theorems~\ref{thm:increments}--\ref{thm:bounded} and Propositions~\ref{prop:pareto}--\ref{prop:bounded_powerlaw} classify and quantify the asymptotics of the solutions to the symmetric LSP for positive, monotonic target densities on both compact intervals and the real line. This solves a major component of the problem for most well-behaved densities.

Two important classes have not been treated carefully: densities with strong, persistent oscillations throughout the tail (or as $x$ approaches the boundary in the case of compact intervals), and densities on compact intervals which approach 0 like a slowly varying function. While we do not attempt to study such cases presently, it seems plausible that a similar approach to the one taken here could provide control on the asymptotics in the case of oscillating densities, provided that $h$ has known upper and lower bounds. 

An interesting direction for future research would be to try to use the rigorous asymptotics as a large-$k$ constraint in a computational algorithm for the optimal search strategy. Early efforts by the author in this direction have been encouraging but have thus far failed to produce convincing evidence of robust convergence properties.

\section*{Acknowledgments}The author thanks Antonio Celani for some early discussions which helped inspire this work. This work was supported by the European Research Council under the grant RIDING (No.\ 101002724), by the Air Force Office of Scientific Research (grant FA8655-20-1-7028), and the National Institute of Health under grant R01DC018789. The author discloses the use of a large language model to help produce some of the technical details of the proofs.
\bibliographystyle{acm}
\bibliography{references}
\end{document}